\documentclass[12pt]{amsart} 
\usepackage{amsmath}
\usepackage{changes}
\usepackage{amsthm}
\usepackage{amsfonts} 
\usepackage{hyperref}
\usepackage[framemethod=TikZ]{mdframed}
\usepackage{epsfig}
\usepackage{amssymb}
\usepackage{mathrsfs}
\usepackage{tcolorbox}
\DeclareMathOperator*{\argmin}{arg\,min}
\mdfdefinestyle{MyFrame}{%
    linecolor=gray!90!white,
    backgroundcolor=gray!0!white}
\numberwithin{equation}{section}

\allowdisplaybreaks

\def\parent#1{#1^{-1}}
\def\r{\varrho}

\newcommand{\eqa}{\begin{eqnarray}}
\newcommand{\ena}{\end{eqnarray}}
\newcommand{\eq}{\begin{equation}}
\newcommand{\en}{\end{equation}}
\newcommand{\eqs}{\begin{eqnarray*}}
\newcommand{\ens}{\end{eqnarray*}}

\def\emptyset{\varnothing} 
\def\ti{\to\infty} 

\def\X{\mathbf{X}}

\def\r{\varrho} 

\def\L{\Lambda} 
 
\newcommand{\Z}     {\mathbb{Z}} 
\newcommand{\N}     {\mathbb{N}} 
\renewcommand{\P}   {\mathbb{P}} 
\newcommand{\E}     {\mathbb{E}}

 \newcommand{\floor}[1]{\left\lfloor #1 \right\rfloor}

\def\1{{\mathchoice {1\mskip-4mu\mathrm l}      
{1\mskip-4mu\mathrm l} 
{1\mskip-4.5mu\mathrm l} {1\mskip-5mu\mathrm l}}} 
\newcommand{\ssup}[1] {{{\scriptscriptstyle{({#1}})}}} 
\def\comment#1{} 
 
\renewcommand{\d}{{\rm d}} 
 
\newcommand{\eps}{\varepsilon}

\newcommand{\Ccal}   {{\mathcal C }}

\newcommand{\Gcal}   {{\mathcal G }}

\def\ignore#1{}
\def\Eq{\ =\ }

\def\Def{:=}

\def\iid{i.i.d.}
\def\Gr{{\rm GR}}
\def\Grm{{\rm GR}^-}

\def\bbE{\mathbb{E}}

\def\bP{\mathbf{P}}
\def\bX{\mathbf{X}}

\def\bE{\mathbf{E}}

\def\parent#1{#1^{-1}}

\def\adb{}

\newtheorem{theorem}{Theorem}
\newtheorem{proposition}[theorem]{Proposition}
\newtheorem{lemma}[theorem]{Lemma}
\newtheorem{remark}[theorem]{Remark}

\newtheorem{definition}[theorem]{Definition}

\newtheorem*{ack}{Acknowledgement}

\renewcommand{\epsilon}{\varepsilon}

\title[Functional CLT for RWRE]{Functional central limit theorem for random walks in random environment defined on regular trees
 }
\date{}
\author[A.~Collevecchio]{Andrea Collevecchio}
\address{Andrea Collevecchio\\ School of Mathematical Sciences, Monash  University, Melbourne} \email{Andrea.Collevecchio@monash.edu}
\author[M.~Takei]{Masato Takei}
\address{Masato Takei\\ Department of Applied Mathematics, Faculty of Engineering, Yokohama National University} \email{takei-masato-fx@ynu.ac.jp}
\author[Y.~Uematsu]{Yuma Uematsu}
\address{Yuma Uematsu\\ Department of Systems Integration, Graduate School of Engineering, Yokohama National University} \email{uematsu-yuuma-zs@ynu.jp}
\date{}
\begin{document}

\begin{abstract}

We study Random Walks in  an i.i.d. Random  Environment (RWRE) defined on $b$-regular trees.~We prove a functional central limit theorem (FCLT)  for transient processes, under a moment condition on the environment. We emphasize that we make no {uniform ellipticity} assumptions.  Our approach relies  on regenerative levels, i.e. levels that are visited exactly once. On the way, we prove that the distance  between consecutive regenerative levels have a geometrically decaying tail.
In the second part of this paper, we apply our results to Linearly Edge-Reinforced Random Walk (LERRW) to prove FCLT when the process is defined on $b$-regular trees, with $ b \ge 4$, substantially improving the results of the first author (see Theorem 3 of \cite{Coll06a}). 

\end{abstract}

\maketitle
\section{introduction}
Random Walk in Random Environment (RWRE)  is a class of {self-interacting processes} that attracted much attention from probabilists since the  seminal work of Kesten, Kozlov, and Spitzer \cite{Kes1} and Solomon \cite{Sol}, in the 70\rq{}s. It seems that  the initial motivation behind this class of process  was related to problems in biology, crystallography and metal physics.  The {interest in this field} grew substantially,   and we refer to \cite{Bog} and \cite{Zei} for an overview of this beautiful subject.
 
We study random walks in  an i.i.d. random  environment defined on $b$-regular trees. We provide a functional central limit theorem (FCLT)  for  processes that are transient, assuming  a moment condition on the environment.   Our approach relies  on regenerative levels, i.e. levels that are visited exactly once. On the way, we prove that the space between these regenerative levels have a geometrically decaying tail. We emphasize that we make no {uniform ellipticity} assumptions.  

{To the best of our knowledge, this is the first  of this type for RWRE on trees {without uniform ellipticity (UE) condition. A FCLT when the environment is UE  is a straightforward consenquence of Proposition 2.2 in \cite{Aid1}, which gives a stretched exponential bound for the regenerative times under UE. This trivial implication was pointed out in \cite{ColS} page 1097 without mentioning that Proposition 2.2 in [2] requires UE. Moreover, Proposition 3.9 in   \cite{ColS}  provides bounds for the covariance {\bf when} the FCLT  holds. The assumptions in Proposition 3.9 are very strict, and comparable to UE. In particular, Linearly Edge-Reinforced Random Walk do not satisfy the assumptions of  Proposition 3.9 in   \cite{ColS}.}  A FCLT was proved by Peres and Zeitouni for biased random walks on Galton-Watson trees (see \cite{Per})}
 
In the second part of this paper, we apply our results to Linearly Edge-Reinforced Random Walk (LERRW), a model introduced in \cite{CD} and which is described below, after Remark~\ref{due}.  
We  prove FCLT for LERRW when the process is defined on $b$-regular trees, with $ b \ge 4$, substantially improving the results of the first author (see Theorem 3 of \cite{Coll06a}).   Moreover, our results  can be combined with the ones of Zhang  \cite{Zhang}  and provide upper large deviations results for $b$-regular trees with $b \ge 4$, which could be improved, with extra computations, to $b \ge 3$ (see Remark~\ref{Zh} below).

Fix an integer $b \in \N$. Let $\Gcal = (V, E)$ be an infinite  $b$-regular tree with root~$\r$.   We augment~$\Gcal$ by adjoining a parent 
$\parent{\r}$ to the root $\r$. In this graph each vertex  has degree $b+1$, with the exception of $\parent{\r}$ that has degree one. If two vertices $\nu$  and $\mu$ are 
the endpoints of the same edge, they are said to be neighbours, and this property is 
denoted by $\nu \sim \mu$. 
The distance $|\nu - \mu|$ between any pair of  vertices $\nu, \mu$, not necessarily adjacent, 
is the number of edges in the unique self-avoiding path connecting $\nu$ to~$\mu$.  For any  other vertex $\nu$, with $\nu \neq \r$,  we let~$|\nu|$ 
be the distance of~$\nu$ from the root~$\r$, i.e.  $|\nu| = |\nu - \r|$. We set $|\parent{\r}| = -1$. 

 We write $\nu < \mu$
if~$\nu$ is an ancestor of~$\mu$, that is if $\nu$ lays on the self-avoiding path connecting $\mu$ to $\r$. Alternatively, we say that $\mu$ is a descendant of $\nu$.  For any vertex $\nu$, denote by $\nu 1, \nu 2, \ldots, \nu b$ its offspring, and by $\parent{\nu}$ its parent.

For $\nu \in V$, let
$$
    \mathbf{A}_{\nu} \Eq (A_{\nu1}, A_{\nu2}, \ldots, A_{\nu b})
$$
to denote the (finite, positive) weights on the edges between $\nu$ and its offspring. 
For simplicity, we index the weight associated to edge~$e$ by the endpoint of~$e$  with larger distance from~$\r$.
The environment~$\omega$ for the random walk on the tree
is then defined, for any vertex $\nu$ with offspring $\nu i$, $1\le i \le b$,  
by the probabilities
\begin{equation}\label{defomega}
  \omega(\nu, \nu i) \Def \frac{A_{\nu i}}{1 + \sum_{1\le j \le b} A_{\nu j}}; 
  \qquad \omega(\nu, \parent{\nu} ) \Def \frac 1{1 + \sum_{1\le j \le b} A_{\nu j}}.
\end{equation}
We set $\omega(\nu,\mu) = 0 $ if $\mu$ and~$\nu$ are not neighbours.
Given the environment $\omega$, we define the random walk 
$\mathbf{X}= \{ X_{n},\, n \ge 0\}$ that starts at~$\r$ to be the Markov chain with
$\mathbf{P}^x_{\omega} (X_{0}=x) =1$, having transition probabilities
$$ 
 \adb{\mathbf{P}^x_{\omega}(X_{n+1} = \mu 1\;|\; X_{n} = \mu) \Eq \omega(\mu, \mu 1).}
$$
whenever $\mu \neq \parent\r$. We set 
$$ \mathbf{P}^x_{\omega}(X_{n+1} = \r \;|\; X_{n} = \parent\r) = 1.$$
The combined probability measure from which the environment is realized is denoted
by~$\mathbb{P}$ and its expectation by $\bbE$, and the semi-direct product 
$\mathbf{P}^x := \mathbb{P} \times \mathbf{P}^x_{\omega}$ represents the annealed measure of the process which starts from vertex $x$.  For simplicity, we use $\mathbf{P}$ and $\mathbf{P}_{\omega}$ respectively for  $\mathbf{P}^{\r}$ and $\mathbf{P}^{\r}_{\omega}$. 
For any vertex $\nu$, set
$$ T_\nu \Def \inf\{ k \ge 0\colon X_k = \nu\}.$$
Sometimes we use $T(\nu)$ instead of $T_\nu$. We are interested in the case (see Assumption A below) where $\bP(T(\nu) = \infty)>0$. Moreover, for $n\in \N$,  let 
$$ T_n \Def  \inf\{ k \ge 0\colon |X_k| = n\}.$$
\begin{mdframed}
[style=MyFrame] 
\noindent{\bf Assumption A} From now on, we suppose that    $\big({\bf A}_\nu\big)_{\nu \in V}$ are \iid,  and
\begin{equation}\label{na0}
\inf_{t \in [0,1]} \E[A^t] > 1/b.
\end{equation}
\end{mdframed}
In particular, condition~\eqref{na0} implies transience of the process, i.e. it visits each vertex only finitely often, a.s.. This result was proved by Lyons and Pemantle \cite{LyoPem}, and see \cite{CollBarb} for a generalization of this result to  Markovian environments.
 A\"id\'ekon (\cite{Aid}, {Theorem 1.5}) proved that  the condition
 \begin{equation}\label{aidcond}
  \E \left[\Big(\sum_{1 \le i \le b} A_{\r i}\Big)^{-1} \right]<\infty,
  \end{equation}
 is sufficient  for  the transient process $\bf{X}$ to have positive speed,   i.e. there exists a positive finite constant $v_b$ such that 
$$ 
\lim_{n \ti} \frac{|X_n|}{n} = v_b, \qquad \bP\mbox{-a.s..}
$$
Denote by  $\floor{x}$ the integer part of $x$. Our main result is the following.
\begin{theorem}[Annealed FCLT] \label{mainth} Under Assumption A,  if we make the  further  assumption
\begin{equation}\label{na1}
  \E \left[\Big(\sum_{1 \le i \le b} A_{\r i}\Big)^{-p} \right]<\infty, \qquad \mbox{ for some $p>2$,}
 \end{equation}
then there exists a  positive constant $\sigma_b$ such that 
\begin{equation}\label{eq:invariance}
\left(\frac{{|X_{nt}| - v_b nt}}{\sqrt{n} \sigma_b}\right)_{t \in [0,1]} \Rightarrow (W_t)_{t \in [0,1]},
\end{equation}
where {$|X_{nt}|$ is the linear interpolation between $|X_{\floor{nt}}|$ and $|X_{\floor{nt}+1}|$ for noninteger values of $nt$,} $(W_t)_t$ is a standard Brownian motion, and $\Rightarrow$ denotes convergence in distribution as $n \to \infty$.
\end{theorem}
\begin{remark}\label{due}
We are not assuming that the random variables  $A_{\r i}$  are bounded or bounded away from $0$, i.e. the so-called {uniform} ellipticity assumption.
\end{remark}
\begin{remark}
{We would like to add few words about the topology under which the convergence in \eqref{eq:invariance} takes place. We consider the space of c\`adl\`ag functions on $[0,1]$ equipped with the Borel $\sigma$-algebra generated by the Skorokhod topology.}
\end{remark}
We apply our results to Linearly Edge-Reinforced Random Walk (LERRW) on trees, which is defined as follows. To each edge of the tree, assign initial weight one. These weights are updated depending on the behaviour of the process. LERRW takes values on the vertices of $\Gcal$, at each step it jumps to vertices which are neighbors of the present one, say $x$. The probability to pick a particular neighbor is proportional to the weight of the edge connecting that vertex to $x$. Each time the process traverses an edge, its weight is increased by one.   See \cite{ST} for a surprising connection between LERRW and the Zirnbauer $H^{2/2}$ model.
 When $\Gcal$ is a tree, we can use a random walk in   i.i.d. random environment to study  LERRW.
LERRW   on the binary tree  is transient and has  positive speed, even though does not satisfy  \eqref{aidcond} {(see \cite{Aid}).} Our result is the following and improves Theorem 3 of \cite{Coll06a}.
\begin{theorem}\label{main:LERRW}
Let $\X$ be LERRW on a $b$-regular tree, with $b\ge 4$. Then $\X$ satisfies \eqref{eq:invariance} for some choice of $\sigma_b$.
\end{theorem}
\begin{remark}\label{Zh} If we replace \eqref{na1} with
\begin{equation} \label{na1'}
  \E \left[\Big(\sum_{1 \le i \le b} A_{\r i}\Big)^{-p} \right]<\infty, \qquad \mbox{ for some $p>1$,}
 \end{equation}
then 
we can prove the finiteness of certain moments of certain regenerative times,
which is enough in order to obtain an upper large deviation result for the speed for the case $b \ge 3$, according to a paper of Zhang \cite{Zhang}.
Notice that the previous known result on this was given by Zhang \cite{Zhang} for  $b\ge 70$. 
\end{remark}

\section{Regenerative times and structure of the proof of Theorem~\ref{mainth}}
From now on, $p$ will be used to denote the exponent that satisfies condition \eqref{na1}.
Under  the assumptions of Theorem~\ref{mainth} (more precisely  \eqref{na0}), the process $\bf{X}$ is transient. 
It is natural to  introduce in this context the so-called regenerative times.
\begin{definition} \label{defregene} 
Set $\tau_0=0$.  For $m\in  \N$ define recursively,
$$\tau_m=\inf\left\{k> \tau_{m-1}: \sup_{j<k}|X_j|< |X_k|\le \inf_{j\ge k}|X_j|\right\}.$$
For each $m\in \Z_+$, let $\ell_m=|X_{\tau_m}|$.
\end{definition}
The elements of the process $(\ell_i)_i$ are called cut levels (or regenerative levels). The  regenerative times $(\tau_i)_i$ are the hitting times of the cut levels.
Under the measure $\bP$, the sequences $((\ell_k - \ell_{k-1},\tau_k- \tau_{k-1}))_{k\ge1}$ are independent and, except for the first one, distributed like $(\ell_1,\tau_1)$ under $\bP\left(\cdot|T(\parent{\r}) = \infty\right)$. 
Moreover, based on a result of Zerner (see  Lemma 3.2.5 in \cite{Zei}), it is not difficult to prove that $\bE[\ell_2 - \ell_1]<\infty$.  We prove that $\ell_2 - \ell_1$, under Assumption A, has an exponential tail. To our knowledge, this result is new.
\begin{theorem}\label{Expotail} Under Assumption A, for any $b \ge 2$, we have that 
$$ \bP(\ell_2 - \ell_1 \ge k) \le a^k,$$
for some constant  $a \in (0, 1)$. 
\end{theorem} 
Afterwards, we  prove that under the assumption \eqref{na1} we have 
\begin{equation}\label{sec-cond-na}
 \bE[(\tau_2 - \tau_1)^2]<\infty.
 \end{equation}
Set 
$Y_i \Def \ell_i -\ell_{i-1} - v_b(\tau_i - \tau_{i-1}).$
We have, for $ \tau_m \le n < \tau_{m+1}$,  
 \begin{equation}\label{na1.1}
  \frac{|X_n| - n v_b}{\sqrt n}  \ge \frac{\ell_m - \tau_{m+1} v_b}{\sqrt{\tau_I}},
 \end{equation}                  
where $I$ equals $m+1$ if the numerator $\ell_m - \tau_{m+1} v_b \ge 0$ and $m$  otherwise.
Hence, 
 \begin{equation}\label{na1.11}
  \frac{|X_n| - n v_b}{\sqrt n}  \ge \sqrt{\frac{m}{\tau_I} }\left(\frac{1}{\sqrt{m}} \sum_{i=1}^m Y_i -  v_b \frac{{\tau_{m}}- \tau_{m+1}}{\sqrt{m}}\right).
 \end{equation}    
 As $\tau_m = \sum_{i=1}^m (\tau_i - \tau_{i-1})$, in virtue of the strong  law of large numbers, $\tau_I/m$ converges a.s. to a positive finite constant. Moreover,  \eqref{sec-cond-na} guarantees that 
 $ \sum_{i=1}^m Y_i/(\sqrt{m})$ weakly converges to  a normal(0, $\sigma$),
 for some finite constant $\sigma>0$,  and  $(\tau_m- \tau_{m+1})/\sqrt{m}$ converges in probability to $0$. Hence, assuming  \eqref{sec-cond-na},  and using Slutzky Lemma, the right hand side of \eqref{na1.11} converges weakly to  a $\mbox{normal}(0,K)$ for some $K \in  (0, \infty)$.
 Similarly  
  \begin{equation}\label{na1.12}
  \frac{|X_n| - n v_b}{\sqrt n}  \le \sqrt{\frac{m+1}{\tau_J} }\left(\frac{1}{\sqrt{{m+1}}} \sum_{i=1}^{{m+1}} Y_i +  v_b \frac{\tau_m- \tau_{m+1}}{\sqrt{{m+1}}}\right),
 \end{equation}    
  where $J$ equals $m$ if ${\ell_{m+1} - \tau_m v_b} \ge 0$ and $m+1$  otherwise. The right-hand side of \eqref{na1.12} converges to a $\mbox{normal}(0, \sigma)$. The procedure to step from the ordinary  central limit theorem to the functional one is classical, and we refer to section 4 of  \cite{DKL02}.

\section{Extension Processes}\label{sec:ext}
 
Here, we define a construction that is closely related to the ones  introduced in \cite{Coll09} and \cite{CKS}. This construction allows to decouple the behaviour of the process on subtrees, even when the process is transient.  This will allow us to build a family of coupled processes which are independent when defined on disjoint subsets of the tree, and usefully correlated to $\mathbf{X}$.

Let $(\Omega, \mathcal{F},\bP)$ denote a probability space on which
\begin{align}\label{defY}
{\bf Y}=(Y(\nu,\mu,k): (\nu,\mu)\in V^2, \mbox{with }\nu \sim \mu, \textrm{ and }{k \in \Z_+})
\end{align}
is a family of independent  exponential random variables with mean 1, and where $(\nu,\mu)$ denotes an {\it ordered} pair of vertices. Below, we use these collections of random variables to generate the steps of $\X$. Moreover,   we  define  a {\it family} of coupled walks using the same collection  of \lq clocks\rq\  $ {\bf Y}$.

Define, for  any $\nu,\mu\in V$ with $\nu\sim \mu$, the quantities
\begin{align} \label{wj1}
r(\nu,\mu)&=  \1_{\{\mu = \parent\nu\}}+\sum_{i=1}^b A_{\nu i} \1_{\{\mu = \nu i\}}.
\end{align}

As it was done in \cite{CKS}, we are now going to define a family of coupled processes on the subtrees of $\mathcal{G}$. For any rooted subtree {$\mathcal{G}'=(V',E')$} of $\mathcal{G}$,
{
the root $\r'$ of $\mathcal{G}'$ is defined as the vertex of $V'$ with smallest distance to $\r$.
Let us define}
the {\it extension} $\X^{ (\mathcal{G}')}$  on $\mathcal{G}'$ as follows.
{Set $X^{ (\mathcal{G}')}_0=\r'$.} For {$\nu \in V'$,} a collection of nonnegative integers $\bar{k}=(k_\mu)_{\mu: [\nu,\mu]\in E'} $, {and $n\ge0$}, let 
\[
\mathcal{A}^{ (\mathcal{G}')}_{\bar{k},n,\nu}=\{X^{ (\mathcal{G}')}_n = \nu\}\cap\bigcap_{\mu: [\nu,\mu]\in E'} \{\#\{1\le j \le n \colon (X^{ (\mathcal{G}')}_{j-1},X^{ (\mathcal{G}')}_j) = (\nu,\mu)\} = k_\mu\}.
\]
Note that the event $\mathcal{A}^{ (\mathcal{G}')}_{\bar{k},n,\nu}$ deals with jumps along oriented edges.
{For} $\nu$, $\nu'$ such that $[\nu, \nu']\in E'$ and for $n\ge0$, on the event 
\begin{align}\label{ursula}
\mathcal{A}^{ (\mathcal{G}')}_{\bar{k},n,\nu}\cap \left\{\nu' = \argmin_{\mu: [\nu,\mu]\in E'}\Big\{\sum_{k=0}^{k_{\mu}}\frac{Y(\nu, \mu, k)}{r(\nu, \mu)} \Big\}\right\}, 
\end{align}
 we set $X^{ (\mathcal{G}')}_{n+1} = \nu'$, where the function $r$ is defined in \eqref{wj1} and the clocks $Y$'s are from the same collection ${\bf Y}$ fixed in \eqref{defY}.
 
We define $\X=\X^{(\mathcal{G})}$ to be the extension on the whole tree.
It is easy to check, from memoryless property  of exponential random variables, that this provides a construction of the RWRE~$\X$ on $\mathcal{G}$.
This continuous-time embedding is classical and it is inspired by {\it Rubin's construction}, after Herman Rubin (see the Appendix in Davis \cite{Dav90}).
If we consider proper subtrees $\mathcal{G}'$ of $\mathcal{G}$, one can check that, with these definitions, the steps of $\X$ on the subtree $\mathcal{G}'$ are given by the steps of $\X^{ (\mathcal{G}')}$. 
Notice that   for any two subtrees $\mathcal{G}'$ and $\mathcal{G}''$ whose edge sets are disjoint, the extensions $\X^{ (\mathcal{G}')}$ and $\X^{ (\mathcal{G}'')}$ are independent as they are defined by two disjoint sub-collections of ${\bf Y}$.

\begin{definition} For  any vertex $\nu \in V$,   define $\rm{fc}(\nu)$, called the first-child of $\nu$, as  the a.s. unique minimizer of {$ Y(\nu, \nu i, 0)/r(\nu, \nu i)$} over the the collection of offspring $(\nu i)_i$ of $\nu$.  For definiteness, the root $\r$ and its parent $\parent\r$ are not  first children. 
\end{definition}
Notice  that a first child is not  necessarily  visited by the process $\bX$. If
the latter 
visits  $\rm{fc}(\nu)$, then it is the first among the children of $\nu$
to be
visited. The random vertices $X_{T_n}$, for $n \ge 1$, are all first children.
\section{Proof of Theorem~\ref{Expotail}}\label{cutlev}

For any vertex $\nu$, with $\nu \neq  \parent\r$, denote by $\L_\nu$ the tree composed by $\nu$, $\parent\nu$, the descendants of $\nu$ and the edges connecting them. This tree is isomorphic to the original tree $\Gcal$. Consider the extension $\X^{\Lambda_\nu}$.  Set 
$$T^{\ssup \nu}_i \Def \inf\{{n > 0} \colon  |X^{\Lambda_\nu}_n| - |\nu| = i\},$$
i.e.  the hitting times of this extensions to level that has distance $i$  from $|\nu|$.
Using condition \eqref{na0}, combined with arguments from  \cite{CollBarb}, for all large $n$  we have 
 \begin{equation}\label{naup1}
 b^{n}\mathbf{P}(T^{\ssup \nu}_{-1} > T^{\ssup \nu}_{n}) >1.
 \end{equation}
{In fact under condition \eqref{na0}, it is proved that 
\begin{equation}\label{naup1.1}
 \lim_{n \ti}  b^{n}\mathbf{P}(T^{\ssup \nu}_{-1} > T^{\ssup \nu}_{n})  =\infty.
 \end{equation}
 }
 (See proof of Theorem 2.1 in \cite{CollBarb}).
Fix $n^*\in \N$ which satisfies \eqref{naup1}. We now construct a branching process as follows. 
We color green the  vertices $\nu$ at level $ n^{*}$ which are visited before ${\bX}$ {returns to $\parent{\r}$}. A vertex~$\nu$ at level~$j n^{*}$, for some integer $j \ge 2$, is colored green, if {\bf both} 
\begin{itemize}
\item its
ancestor~$\mu$ at level $(j-1) n^{*}$ is green,  {\bf and} 
\item the extension over the path $[\parent{\mu}, \nu]$, hits $\nu$ before returning to $\parent{\mu}$.
\end{itemize}
The green vertices evolve as a Galton--Watson tree, with offspring mean
$b^{n^{*}}\mathbf{P}(T^{\ssup \nu}_{-1} > T^{\ssup \nu}_{n^{*}}) > 1$.  Hence this random tree is supercritical, and thus
the probability of there being an infinite number of green vertices is positive.  

Denote by ${\rm GR}(\L_{\nu})$ the set of green vertices on the tree $\L_{\nu}$.
Next, we want to define a sequence of events which we show to be independent and which are closely related
to the event that a given level is a cut level. Fix {$\xi \in \N$}. For any vertex
$\nu \in V$, define the random subset of vertices $\Theta_\nu \subset V$ as follows. Vertex $\mu \in \Theta_\nu$, iff
\begin{itemize}
\item $\mu$  is a descendant of $\nu$;
\item  the distance between $\mu$ and $\nu$ is a multiple of $\xi n^*$;
\item $\mu$ is a first child.
\end{itemize}
\begin{definition} Define 
$\Grm(\L_{\nu})$ to be the set of green vertices obtained from $\Gr(\L_{\nu})$ by deleting elements of $\Theta_{\nu}$ {and their descendants.}
Define the event 
$$ E(\nu) \Def \{|\Grm(\L_{\nu})| = \infty\}.$$
\end{definition}
\begin{proposition} For any $\xi$ large enough, $\bP\big(E(\nu)\big)>0$. 
\end{proposition}
\begin{proof}
{As we observed, the green vertices evolve as a supercritical Galton--Watson tree.  The event $E(\nu)$ is the survival event for a certain subtree of the Galton-Watson tree of green vertices, obtained by pruning. Choose $\xi$ large enough that 
 \begin{equation}\label{brownsup}
  b^{\xi n^* -1}  (b-1) \mathbf{P}(T^{\ssup \nu}_{-1} > T^{\ssup \nu}_{\xi n^*}) >1.
  \end{equation}
 This is possible because of \eqref{naup1.1}. 
  Color brown the descendants of $\nu$ { at} distance $\xi n^*$ from $\nu$ which are green { vertices} and are not first children. Recursively, color brown the descendants of $\nu$ at distance $k \xi n^*$, 
  \begin{itemize}
  \item { which are green vertices and not first children, and}
  \item whose ancestors { at level $ (k-1) \xi n^*$ from $\nu$} are brown.
 \end{itemize}
 The random set of vertices of brown vertices evolve like the population of a branching process with mean offspring larger than 
 $$ b^{\xi n^* -1}  (b-1) \mathbf{P}(T^{\ssup \nu}_{-1} > T^{\ssup \nu}_{\xi n^*}).$$
Hence  it is supercritical  by virtue of \eqref{brownsup}, which in turn implies that $\bP\big(E(\nu)\big)~>~0$. 
}
\hfill 
\end{proof}
\begin{proposition}
The events {$D_k \Def E(X_{T_{k\xi n^*}})$}, with $k \in \N$, are independent. 
\end{proposition}
\begin{proof}
Fix indices $i_1< i_2 < \ldots< i_n$.  It is enough to prove that  
\begin{equation}\label{na:dk} \bP\left(\bigcap_{k=1}^n D_{i_k}\right)  = \bP\left(\bigcap_{k=1}^{n-1} D_{i_k} \right) \bP(D_{i_n}).
\end{equation}
To prove \eqref{na:dk}, we condition on the possible values of {$X_{T_{i_n \xi n^*}}$}. To simplify notation, set $\iota_n = i_n \xi n^*$. 
$$
\begin{aligned}
\bP\left(\bigcap_{k=1}^n D_{i_k}\right)= \sum_{\nu \in V\colon |\nu| = \iota_n} \bP\left(\bigcap_{k=1}^n D_{i_k} \;|\; {X_{T_{\iota_n}} = \nu}\right) \bP\left({X_{T_{\iota_n}} = \nu}\right). 
\end{aligned}
$$
Conditionally on $\{{X_{T_{\iota_n}} = \nu}\}$, $D_{i_n}$ is determined by the collection of exponentials {$Y(x, y, k)$} where both $x, y$ are vertices of $\L_\nu$ and {$k \in \N$}. On the other hand, conditionally on  $\{{X_{T_{\iota_n}} = \nu}\}$, the event $\bigcap_{k=1}^{n-1} D_{i_k}$ depends on a disjoint set of exponentials. Hence $D_{i_n}$ and $\bigcap_{k=1}^{n-1} D_{i_k}$  are, given  $\{{X_{T_{\iota_n}} = \nu}\}$, conditionally  independent, i.e. 
$$ \bP\left(\bigcap_{k=1}^n D_{i_k}\;|\; {X_{T_{\iota_n}} = \nu}\right) = \bP\left(\bigcap_{k=1}^{n-1} D_{i_k}\;|\; {X_{T_{\iota_n}} = \nu}\right) \bP(D_{i_n}\;|\;{X_{T_{\iota_n}} = \nu}).$$
Finally, notice that by a simple symmetry argument, we have
$$ \bP(D_{i_n}\;|\;{X_{T_{\iota_n}} = \nu}) = \bP(D_{i_n}).$$ 
\hfill 
\end{proof}

\begin{proof}[Proof of Theorem~\ref{Expotail}]
First, notice that if {$D_i$} holds, then {$i \xi n^*$} is a cut level, as after time {$T_{i \xi n^*}$}, the process $\bX$ never visits level {$i \xi n^*-1$} again. 
Hence, 
$$ \bP(\ell_{1} \ge n) \le {\bP\left(\bigcap_{i=1}^{\floor{n/(\xi n^*)}} D^c_i\right) \le  \bP\Big( D^c_1\Big)^{\floor{n/(\xi n^*)}},}$$
proving the theorem.
\hfill
\end{proof}

\section{Finite second moment between cut times}

Our plan is to prove, in order, that 
\begin{itemize}
\item the $q$-th moment of the number of distinct vertices visited by time $\tau_1$ grows as a power function, with degree $q$, for all $q>1$.
\item The $p$-th moment of the total number of visits to $\r$ by the process $\bX$ is finite. 
\end{itemize}
Let $\Pi_n$ be the cardinality of the range, that is i.e. the number of distinct vertices,  of $\X$ by time $T_n$. The number of vertices visited at level $i$ is bounded {by $Y_i$}, where $(Y_i)_i$ is a sequence of \iid $\;$geometric random variables. 
We recall that the process is transient.  Hence,  each  time it jumps to an unvisited vertex $\nu$  there is a fixed positive  probability  that the process never visits again $\parent{\nu}$. Hence for any $q >1$,
$$ \bE[\Pi_n^q] \le  \bE\left[\left(\sum_{i=1}^{n} Y_i\right)^q\right] \le n^q   \bE[Y_1^q].$$
In other words,  for any $q >1$, we have
\begin{equation}\label{naup2} 
 \bE[\Pi_n^q] = O(n^q).
\end{equation}
Define $\Pi = \Pi_{\ell_1}$, that is the number of different vertices visited by the time the process hits the first cut level.
\begin{lemma}For any $q >1$, we have that 
\begin{equation}\label{naups2}
 \bE[\Pi^q] <\infty.
 \end{equation}
\end{lemma}
\begin{proof}
Notice that 
$$ \bE[\Pi^q] = \sum_{n=0}^\infty \bE\left[\Pi_n^q \1_{\{\ell_1 =n\}}\right] \le C \sum_{n=0}^\infty n^q \bP(\ell_1 \ge n)^{1/2}<\infty,$$
where we used  Cauchy-Schwarz\rq{}s inequality, \eqref{naup2} and Theorem~\ref{Expotail}.
\hfill
\end{proof}

Define, for $ \nu \in V$,
$$
L_\nu  \Def {\sum_{k=0}^{\infty}} \1_{\{X_k = \nu\}},   \qquad  \beta_\nu(\omega) \Def \mathbf{P}^\nu_{\omega}(T_{\parent{\nu}} = \infty), 
$$
which respectively are,  the total time spent in $\nu$ and the quenched probability that the walk never returns to $\parent{\nu}$.
\begin{remark}\label{remna1} Under the measure $\bP_\omega$, the random variable $L_{\parent\r}$ is distributed as Geometric($\beta_\r (\omega)$), i.e.
$$ {\bP_\omega}(L_{\parent\r} = k) = \Big(1- \beta_\r (\omega)\Big)^k \beta_\r (\omega) \qquad \mbox{for $ k \ge 0$}.$$
\end{remark}
\begin{proposition}\label{prop:ku1}
Under the assumptions of Theorem~\ref{mainth} we have that 
$$ \E\left[\Big(\beta_\r\Big)^{-p}\right] < \infty,$$
where $p>2$ satisfies \eqref{na1}.
\end{proposition}
\begin{proof}
This proof is  inspired by the proof of Lemma 2.2 in \cite{Aid}.  We include the steps for completeness.  
\begin{equation}\label{na9}
 \beta_\nu(\omega) = \sum_{i=1}^b \omega(\nu, \nu i) \beta_{\nu i}(\omega) + \sum_{i=1}^b  \omega(\nu, \nu i) (1- \beta_{\nu i}(\omega)) \beta_\nu(\omega).
 \end{equation}
From \eqref{na9}, it follows that 
\begin{equation}\label{na9.1}
 \frac 1{\beta_\nu} = 1+ \frac1 {\sum_{i=1}^b A_{\nu i} \beta_{\nu i}} \le 1 + \min_{1 \le i \le b} \frac 1{A_{\nu i} \beta_{\nu i}}.
 \end{equation}
Consider a random path generated as follows. We set $v_0 = \r$, and we define $v_k$, with $k\ge 1$, recursively. Suppose that $v_j$ for $j \le k$  are defined. Set $v_{k+1}$ to be one of the maximizers $x \mapsto A_x$, where $x$ ranges over the offspring of $v_k$.  If there is more than one maximizer, we choose among them uniformly at random.  Define  $\Ccal(v_k)$ the set of offspring of $v_k$ different from  $v_{k+1}$. Fix $\eps >0$.
$$E_n \Def { \bigcap_{k=1}^{{n}}} \bigcap_{y \in \Ccal(v_k)} \Big\{ A_y \beta_y < \eps\Big\}.
$$
We set $E_0^c = \emptyset$. 
{Notice that  $E_{n+1}\subset E_n$ and that on the event $ E^c_{n+1}\cap E_n$ we have
$$ \min_{y \in  { \Ccal(v_n)} \cup \{v_{n+1}\}}  \frac {1 }{ A_y\beta_y(\omega)} \le \frac 1\eps.$$
Combining these two facts with  \eqref{na9.1},
  we infer the following. Denote  by $(y(i))_i$ the offspring of $v_n$,
\begin{equation}\label{na9.11}
\begin{aligned}
\frac{\1_{E_n}}{ \beta_{v_n}(\omega)} &
\le  1 + \min_{i \in \N} \frac {\1_{E_{n+1}^c} \1_{E_n} }{ A_{y(i)}\beta_{y(i)}(\omega)}  +\min_{i \in  \N} \frac {\1_{E_{n+1}}}{ A_{y(i)}\beta_{y(i)}(\omega)}\\
&\le  1 +  \frac 1 \eps\1_{E_{n+1}^c} \1_{E_n} + \min_{y \in  { \Ccal(v_n)}} \frac {\1_{E_{n+1}}}{ A_y \beta_y(\omega)}\\
&\le 1 +  \frac 1 \eps\1_{E_{n+1}^c} \1_{E_n} + \frac 1 \eps\1_{E_{n+1}} +  \frac{\1_{E_{n+1}}}{ A_{v_{n+1}} \beta_{v_{n+1}}(\omega)}\\
&\le 1+ \frac 1 \eps + \frac{\1_{E_{n+1}}}{ A_{v_{n+1}} \beta_{v_{n+1}}(\omega)}.
\end{aligned}
\end{equation}
Following  \cite{Aid} (proof of Lemma 2.2), by the i.i.d. structure of the environment, we have} $\P(E_n) = \P(E_1)^n$  . By reiterating \eqref{na9.11}, we obtain 
\begin{equation}\label{na2}
\frac{1}{\beta_\r} \le \Big(1 + \frac 1{\eps}\Big) \left(1 + \sum_{n=1}^\infty B(n)\right),
\end{equation}
where $B(n) = \1_{E_n} \prod_{j=1}^{n} \big(A_{v_j}\big)^{-1}$.   
Notice that for any sequence of non-negative numbers $(b_n)_n$,   we have 
\begin{equation}\label{na3}
{
\left( \sum_{n=1}^\infty b_n \right)^p \le \sum_{n=1}^\infty 2^{(p-1)n}  b_n^p.
}
\end{equation}
In order to prove \eqref{na3}, it is enough to notice that 
$$ {\sum_{n=1}^\infty b_n =  \sum_{n=1}^\infty 2^{-n} 2^n b_n,} $$
and apply  Jensen's inequality.  
Next, we combine \eqref{na2} and \eqref{na3},  to get 
\begin{equation}\label{na4}
\begin{aligned}
 \frac{1}{\beta_\r^p}  &\le \Big(1 + \frac 1{\eps}\Big)^p   \left(1 + \sum_{n=1}^\infty B(n)\right)^p\\
 &\le \Big(1 + \frac 1{\eps}\Big)^p 2^{p-1} \left(1 + \sum_{n=1}^\infty 2^{(p-1)n}  B(n)^p \right).
 \end{aligned}
\end{equation}
We have
$$\E[B(n)^p] = \E\left[\1_{E_1} \Big(A_{v_1}\Big)^{-p}\right]^n =: c^n.$$ 
In virtue of the definition of $E_1$ and the integrability condition \eqref{na1}, we  can choose $\eps$ small enough that $c < 1/2^{p-1}$. Hence, by taking expectations in \eqref{na4}, we get
$$ \E\left[\frac{1}{\beta_\r^p}\right] < \infty.$$
\hfill
\end{proof}

\begin{proposition} We have
\begin{equation}\label{na5}
\bE\left[L_\r^p\right] < \infty.
\end{equation}
\end{proposition}
\begin{proof}
 In virtue of Remark~\ref{remna1},  combined with  Proposition~\ref{prop:naap} in the Appendix, we have 
\begin{equation}\label{na8}
\begin{aligned}
 \bE[L_{\parent\r}^p] &\le 1+ C_p \E\left[ \frac{ \beta_\r}{ (- \ln(1 - \beta_\r))^{p+1}}\right]\\
 & \le 1+ {C_p} \E\left[ \frac{1}{ \beta_\r^{p}}\right]<\infty,
 \end{aligned}
 \end{equation}
where in the last step we used the fact that
{$x/(-\ln(1 - x)) \in (0,1)$ for any $x \in (0,1)$,} and Proposition~\ref{prop:ku1}. Recall that the subtrees $\Lambda_\nu$ was introduced at the beginning of Section~\ref{cutlev}.  
Next, denote by   $\Lambda_i$ the tree composed by $\r$, $\r i$, all the descendants of $\r i$, and the edges connecting them.
Denote by $\widetilde{L}_i$ the number of visits to $\r$ by the extension $\mathbf{X}^{\L_i}$, i.e.
$$ \widetilde{L}_i \Def \sum_{k=0}^\infty \1_{\{X^{\L_i}_k = \r\}}.$$
Under the measure $\bP$, $\widetilde{L}_i$ and $L_{\parent \r}$ are equally distributed.  Hence, $\bE[\widetilde{L}_i^p] < \infty$. Finally
$$ L_\r \le 1 + L_{\parent \r} + \sum_{i=1}^b \widetilde{L}_i,$$
proving our result.
\hfill \end{proof}
\begin{lemma} There exists a collection of random variables $(\bar{L}_{\nu})_\nu$, such that  $\bar{L}_{\nu} \sim L_\r$ and $L_{\nu}  \le \bar{L}_{\nu},$ a.s.,  for all $\nu \in V$, with $\nu \neq \parent{\r}$.
\end{lemma}
\begin{proof}
Consider the extension $\bX^{\Lambda_\nu}$. Set
\begin{equation}\label{barl} \bar{L}_{\nu} \Def \sum_{k=0}^\infty \1_{\{X^{\Lambda_\nu}_k = \nu\}}.
\end{equation}
By the definition of the extension, $\bar{L}_{\nu}$ { shares the same  distribution as $L_\r$. Moreover, $\bar{L}_{\nu} \sim L_\r $ as $\bX$ observed while in $\L_\nu$ and the extension coincide up to the last time the former process leaves $\L_\nu$ and never returns to it.}
\hfill
\end{proof}
{Let $(\sigma_i)_i$  be the sequence vertices visited by the process $\bX$, ordered chronologically. More precisely, $\nu = \sigma_i$ if and only if  $\nu$ is the $i$-th  distinct vertex visited by  $\bX$ and $\ell_1 > |\nu|$.}
{
\begin{proposition}\label{propls} For any $p>0$, there exists    $C_p \in (0, \infty)$, such that 
$$ \sup_{k \in \N} \bE\left[L_{\sigma_k}^p\right] < C_p.$$
\end{proposition}}
{\begin{proof}
Consider the random variable $\bar{L}_{\nu}$ defined in \eqref{barl}. Notice that $\bar{L}_{\nu}$ and the event $\{\sigma_k = \nu\}$ are independent.  In fact the latter is generated  using the collection of Poisson processes  
$$ \{ Y(x, y)\colon \mbox{none of $x$ and $y$ are descendants of $\nu$}\},$$
where the processes $Y$ were introduced in 
\eqref{defY} and each vertex $\nu$ is NOT considered a descendant of itself, while the extension is generated using a disjoint collection of Poisson processes. 
Hence 
$$ 
\begin{aligned}
\bE\left[L_{\sigma_k}^p\right] &=  \bE\left[\sum_{\nu} L_{\nu}^p \1_{\sigma_k = \nu}\right]\le \sum_{\nu} \bE\left[ \bar{L}_{\nu}^p \1_{\sigma_k = \nu}\right]\\
&= \sum_{\nu} \bE\left[ \bar{L}_{\nu}^p\right] \P(\sigma_k = \nu) \\
&= \bE\left[ \bar{L}_{\r}^p\right]<\infty.
\end{aligned}
$$ 
Hence, by taking $C_p = \bE\left[ \bar{L}_{\r}^p\right]$ we conclude our proof. \hfill
\end{proof}}
\begin{proposition} \label{prop:Proposition15}
$\bE[\tau_1^2]<\infty.$
\end{proposition}

\begin{proof}
The vertices  $\sigma_1, \sigma_2, \ldots, \sigma_{\Pi}$ are the vertices visited before time $\tau_1$. We have
\begin{equation}\label{na10}
\begin{aligned}
\bE[\tau_1^2] \le \bE\left[\left(\sum_{k=1}^\Pi L_{\sigma_k}\right)^2\right]\le \bE\left[ \Pi \sum_{k=1}^\Pi L_{\sigma_k}^2\right].
\end{aligned}
\end{equation}
Let $q>1$ the conjugate of $p/2$, i.e. $(1/q)+ (2/p)=1$.  By H\"older's inequality, {and using Proposition~\ref{propls}}, we have that
\begin{equation}\label{na11}
\begin{aligned}
 \bE\left[ \Pi \sum_{k=1}^\Pi L_{\sigma_k}^2\right] &=\bE\left[ \sum_{k=1}^\infty L_{\sigma_k}^2 \Pi \1_{\{\Pi \ge k\}}\right]\\
 &\le \sum_{k=1}^\infty {\bE\left[L_{\sigma_k}^p\right]^{2/p}} \bE\left[\Pi^q\1_{\{\Pi \ge k\}}\right]^{1/q}\\
 &\le {(C_p)^p} \sum_{k=1}^\infty\bE\left[\Pi^q\1_{\{\Pi \ge k\}}\right]^{1/q} <\infty,
 \end{aligned}
\end{equation}
{where $C_{p}$ is the same as in Proposition~\ref{propls}.}
\hspace{\fill} \end{proof}
{Finally, notice that 
$$ \bE[(\tau_2 - \tau_1)^2] = \bE[\tau_1^2\;|\; T_{\parent\r} = \infty] \le  \frac{\bE[\tau_1^2]}{\bP(T_{\parent\r} = \infty)} <\infty,$$
which proves \eqref{sec-cond-na}.}
\section{Linearly Edge-Reinforced Random Walks}
 Define a discrete time process $\bX$ as follows. It takes as values the vertices of $\Gcal$. Initially all the edges are given weight one, and $X_0 = \r$. Whenever an edge is traversed, i.e. the process jumps from one of its endpoints to the other, the weight of the edge is increased by one, and the process jumps to nearest neighbors with probabilities proportional to the weights of the connecting edges. Notice, that this process can be represented as a RWRE with the following environment (see  \cite{Pemtree}).
 To each vertex $\nu \neq \parent \r$, assign  an independent $(b+1)$-dimensional random vector $\mathbf{Z}_\nu = (Z^{\ssup \nu}_0, Z^{\ssup \nu}_1, \ldots, Z^{\ssup \nu}_b)$, distributed as a Dirichlet distribution, with parameters $(1, 1/2, 1/2, \ldots, 1/2)$.
{The distribution of the vector assigned to $\r$ is still a Dirichlet distribution but with different parameters. This exception does not affect our analysis, as we are interested in a limit theorem.}
  Set 
 $$ A_{\nu i} = \frac{Z^{\ssup \nu}_i}{Z^{\ssup \nu}_0}.$$
Theorem \ref{mainth} can be applied to this process to yield a functional central limit theorem for $b \ge 5$. In fact, condition \eqref{na0} is satisfied (see Pemantle \cite{Pemtree}).
We have 
$$ 
\begin{aligned}
\E \left[\Big(\sum_{1 \le i \le b} A_{\r i}\Big)^{-p} \right] &= \E \left[\Big(\frac{Z_0^{\ssup \r}}{1- Z_0^{\ssup \r}}\Big)^p\right]\\
 &= \int_0^1   \frac{\Gamma(\frac{b}{2}+1)}{\Gamma(\frac{b}{2})} \left(\frac{x}{1-x}\right)^p (1-x)^{\frac{b}{2} -1} \d x \\
 & = \dfrac{B(1+p,\frac{b}{2}-p)}{B(1,\frac{b}{2})},
\end{aligned}
$$
which is finite if and only if $b/2-p>0$, i.e. $b>2p>4$.
Hence condition \eqref{na1} is satisfied for $b \ge 5$.

\begin{remark} As in \cite{Pemtree}, let us consider the more general situation where each time the process traverses an edge, its weight is increased by $\Delta>0$. In this case $\mathbf{Z}_\nu = (Z^{\ssup \nu}_0, Z^{\ssup \nu}_1, \ldots, Z^{\ssup \nu}_b)$ is distributed as a Dirichlet distribution with parameters $((1+\Delta)/(2\Delta), 1/(2\Delta), 1/(2\Delta), \ldots, 1/(2\Delta))$, and
$$ 
\begin{aligned}
\E \left[\Big(\sum_{1 \le i \le b} A_{\r i}\Big)^{-p} \right]
 & = \dfrac{B(\frac{1+\Delta}{2\Delta}+p,\frac{b}{2\Delta}-p)}{B(\frac{1+\Delta}{2\Delta},\frac{b}{2\Delta})}.
\end{aligned}
$$
Thus condition \eqref{na1} is satisfied if and only if $0<\Delta<b/4$.
\end{remark}


Now we turn to the proof of Theorem~\ref{main:LERRW} for the case $b=4$. For any $\nu \in V$, with $|\nu| \ge 0$, define
\begin{align*}
T_{\nu}^+  &\Def \inf\{k>0 \colon X_k = \nu\}, \qquad
\gamma_\nu (\omega) \Def\bP^\nu_{\omega}(T_{\parent{\nu}} = \infty, {T^+_{\nu}} = \infty).
\end{align*}
{
Recall the definition of $\omega(\nu, \nu i)$ given in \eqref{defomega}, and that for LERRW on the $b$-regular tree and for $|\nu| \ge 1$, we have $\omega(\nu, \nu i)$ is  distributed as Beta$(1/2, (b+1)/2)$, while   $\omega(\nu, \parent{\nu})$ is a Beta$(1, b/2)$.} The reasoning presented in this section follows closely the one given in  section 7 in \cite{Aid}. 
\begin{remark}\label{assumptionL}
We  fix $\epsilon \in (0, 1/3)$, and  we can assume $\omega(\r,\parent{\r})\leq 1-\epsilon$ without loss of generality.  In fact, our goal is to prove a  limit theorem, and the process is transient. Hence the influence of $\omega(\r,\parent{\r})$ vanishes in the limit. 
\end{remark}
\begin{proposition}\label{prop:Yuuma4-1}
For LERRW on the 4-regular tree, {there exists a positive finite constant $C$ such that} for any vertex  $\nu$ with $|\nu| \ge 0$, we have
\[
\E\left[\left(\frac{\1_{\{\omega(\nu, \parent{\nu})\leq 1-\epsilon\}}}{\gamma_\nu (\omega)}\right)^{{\frac {28}{9}}}\right] {\leq C}.
\] 
\end{proposition}
\begin{proof} From now on, fix $ \delta =35/18$, $\chi = 9/28$, and $\alpha = 71/280$. From the proof of Lemma 7.1 in \cite{Aid}, we have
\begin{equation}\label{betad}
\E\left[(\beta_\nu)^{-\delta}\right]<\infty.
\end{equation}
Moreover
\[
\frac{1}{\gamma_\nu (\omega)}=\frac{1}{\sum_{i=1}^{4}\omega(\nu, \nu i)\beta_{\nu i}}\leq \min_{1\leq i \leq 4}\frac{1}{\omega(\nu, \nu i)\beta_{\nu i}}.
\]
{By Fubini's theorem, we have
\begin{align*}
&\E\left[\left(\frac{\1_{\{\omega(\nu, \parent{\nu})\leq 1-\epsilon\}}}{\gamma_\nu (\omega)}\right)^{\frac{1}{\chi}}\right] \\
&\leq \E\left[\left(\min_{1\leq i \leq 4}\frac{1}{\omega(\nu, \nu i)\beta_{\nu i}}\right)^{\frac{1}{\chi}}\1_{\{\omega(\nu, \parent{\nu})\leq 1-\epsilon\}} \right] \\
&= \int_0^{\infty} \P \bigl(\omega(\nu, \parent{\nu})\leq 1-\epsilon, \min_{1\leq i \leq 4}(\omega(\nu, \nu i)\beta_{\nu i})^{-1/\chi} \geq n \bigr) \d n.
\end{align*}
Notice that}
\begin{align*}
 \{(\omega(\nu, \nu i)\beta_{\nu i})^{-1} {\geq} n^\chi \} &\subset \{\omega(\nu, \nu i)^{-1} {\geq} n^\alpha \}\cup\{(\beta_{\nu i})^{-1} {\geq} n^{\chi-\alpha} \} \\
 &=\colon E_i^1 \cup E_i^2
\end{align*}
for each of $i$. 
  On the event $\{\omega(\nu, \parent{\nu})\leq 1-\epsilon\}$, there exists $1\leq i \leq 4$ such that $\omega(\nu, \nu i)\geq \epsilon /4$.
By symmetry,
\begin{equation}\label{Yu4}
\begin{aligned}
&\P \bigl(\omega(\nu, \parent{\nu})\leq 1-\epsilon, \min_{1\leq i \leq 4}(\omega(\nu, \nu i)\beta_{\nu i})^{-1/\chi} {\geq} n \bigr)\\
&\leq 4\P \bigl(\omega(\nu, \nu 4)\geq \epsilon /4, \min_{1\leq i \leq 4}(\omega(\nu, \nu i)\beta_{\nu i})^{-1} {\geq} n^\chi \bigr)   \qquad \mbox{(union bound)}\\
&= 4\P \bigl(\omega(\nu, \nu 4)\geq \epsilon /4,\, (\beta_{\nu 4})^{-1}{\geq} n^\chi \omega(\nu, \nu 4), \min_{1\leq i \leq 3}(\omega(\nu, \nu i)\beta_{\nu i})^{-1} {\geq} n^\chi \bigr)\\
&\leq 4\P \bigl((\beta_{\nu 4})^{-1}{\geq} n^\chi \epsilon /4, \min_{1\leq i \leq 3}(\omega(\nu, \nu i)\beta_{\nu i})^{-1} {\geq} n^\chi \bigr)\\
&\leq 4\sum_{(k_1, k_2, k_3)\in \{1, 2\}^3}\P((\beta_{\nu 4})^{-\delta}{\geq} n^{\delta\chi} (\epsilon /4)^{\delta}, \cap_{i=1}^{3}E_i^{k_i})\\
&= 4\sum_{(k_1, k_2, k_3)\in \{1, 2\}^3}\P((\beta_{\nu 4})^{-\delta}{\geq} n^{\delta\chi} (\epsilon /4)^{\delta}) \P(\cap_{i=1}^{3}E_i^{k_i})\qquad \mbox{(independence)}\\
&\le 4c_0 n^{-\delta\chi} \sum_{(k_1, k_2, k_3)\in \{1, 2\}^3} \P(\cap_{i=1}^{3}E_i^{k_i}) \qquad\mbox{(Markov's ineq. and  \eqref{betad})} \\
&=4c_0 n^{-\delta\chi} \left[\P(E_1^2)^3+3\P(E_1^2)^2 \P({E_3^1}) + 3\P(E_1^2)\P(E_2^1\cap E_3^1) + \P(\cap_{i=1}^3 E_i^1) \right].
\end{aligned}
\end{equation}
{In the last equality we used independence and symmetry.}
We have $\P(E_1^2)\leq c_1 n^{-\delta (\chi - \alpha)}$ again by Markov's inequality. Since  $\omega(\nu, {\nu 3})$ 
is a Beta($1/2, 5/2$), we have $\P({E_3^1}) \leq c_2 n^{-\alpha /2}$. Notice that conditionally to $\omega(\nu, {\nu 3})$, the random variable  $\omega(\nu, {\nu 2})/\{1-\omega(\nu, {\nu 3})\}$ is distributed as a Beta($1/2, 2$).  Hence
\[ \P(E_2^1\cap E_3^1)={\P(E_3^1)\P(E_2^1|E_3^1)}  \leq c_3 n^{-\alpha}. \]
In the same way, we have $\P(\cap_{i=1}^3 E_i^1)\leq c_4 n^{-3\alpha /2}$. Therefore,  using  \eqref{Yu4}, we have
\begin{align*}
\P \bigl(\omega(\nu, \parent{\nu})\leq 1-\epsilon, \min_{1\leq i \leq 4}(\omega(\nu, \nu i)\beta_{\nu i})^{-1/\chi} {\geq} n \bigr)
\leq\sum_{i=0}^{3}c'_i n^{-\delta\chi} n^{-\delta (\chi - \alpha)i} n^{-(3-i)\alpha/2}.
\end{align*}
With our choice of $\alpha, \delta$ and $\chi$, we have $\delta\chi + 3\alpha /2 >1$  and $\delta\chi + 3\delta (\chi - \alpha)>1$, and this completes the proof. 
\hspace{\fill}
\end{proof}
\begin{proposition}\label{prop:Yuuma4-2}
For LERRW on the 4-regular tree,
$\bE[{L_{y}^3}]<\infty$
for any $  y\in V$.
\end{proposition}

\begin{proof}
Define, for $\nu \in V$ and $t \geq 0$,
\begin{align*}
L_{\nu}(t)&\Def \sum_{k=0}^t \1_{\{X_k =\nu\}}, \\
L^+_{\nu}(t)&\Def \sum_{k=0}^{t-1} \1_{\{(X_k, X_{k+1})=(\nu^{-1}, \nu)\}},\\
L^-_{\nu}(t)&\Def \sum_{k=0}^{t-1} {\1_{\{(X_k, X_{k+1})=(\nu, \nu^{-1})\}}.}
\end{align*}
Noting that $L^-_{\nu}(t) \leq L^+_{\nu}(t)$ for $\nu \neq \r$, we have
\[
L_\nu(t)=\1_{\{\nu =\rho \}}+L^+_\nu(t)+\sum_{i=1}^b L_{\nu i}^-(t)\leq 1+L^+_\nu(t)+\sum_{i=1}^b L^+_{\nu i}(t).
\]
Let $L_{y}^+ \Def L^+_y(\infty)$. To prove $\bE[{L_{y}^3}]<\infty$, it is enough to show $\bE[(L_{y}^+)^3]<\infty$.

Recall that $\omega(\r,\parent{\r})\leq 1-\epsilon$ (see Remark~\ref{assumptionL}).
For  a vertex $y$,  and let   $Y$ be the youngest ancestor of $y$, with $\omega(Y,\parent{Y})\leq 1-\epsilon$. More precisely $Y$ is the vertex $z$ on the unique self-avoiding path connecting $y$ to $\r$,  which has maximum distance from $\r$ and satisfies $\omega(z,\parent{z})\leq 1-\epsilon$. We have $\omega(g,\parent{g})> 1-\epsilon$ for all  ancestors $g$ of $y$ with $|Y| < |{g}| < |y|$. 
Let $T_y^Y :=\inf\{i > T_Y   \colon X_{i} =y \}$, that is the hitting time of $y$ after $T_Y$. Notice that, on $\{X_0=y\}$,  
\begin{equation}\label{yuu1}
L_y^+=L_y^+(T_Y)+ \1_{\{T_Y <\infty\}} \1_{\{T_y^Y<\infty \}} \widetilde{L}_y^+,
\end{equation}
where $\widetilde{L}_y^+$ is an independent copy of $L_y^+$. Coupling with the simple random walk on the path $[Y, y]$, {we prove next that 
\begin{equation}\label{finpo1}
 \bP_{\omega}^y (T_y^+ < T_Y)\leq\dfrac{2}{3}.
 \end{equation}
In fact, if $Y \notin \{y, \parent{y}\}$, we have that 
$$ \bP_{\omega}^y (T_y^+ < T_Y) \le \max_{a \in [0,\eps]} \left(a \cdot  1 +  (1- a)  \frac \eps {1 -\eps}\right),
$$
where the term $ \eps/({1 -\eps})$ is derived using a coupling with a biased random walk.
As $\eps/(1- \eps) <1$, we have that the maximum is attained at $a = \eps$, which implies 
$$ \bP_{\omega}^y (T_y^+ < T_Y) \le  \epsilon + (1-\epsilon) \cdot \frac{\epsilon}{1-\epsilon}=2\epsilon \leq  \dfrac{2}{3},$$
proving
\eqref{finpo1} in the case $Y \notin \{y, \parent{y}\}$. If $Y=\parent{y}$, then 
$$ \bP_{\omega}^y (T_y^+ < T_Y) \le \epsilon \le \frac 23,$$
 while if $Y = y$ then 
$$ \bP_{\omega}^y (T_y^+ < T_Y) = 0 \le \frac 23.$$
}
Hence 
$L_y^+(T_Y)$ is stochastically dominated by a geometric distribution with average  3, that is ${\bE_{\omega}^y} [L_y^+(T_Y)]\leq 1/(1-2/3)=3$. This implies that  there exist positive constants $c_5, c_6$ such that 
\[ {\bE_{\omega}^y} [L_y^+(T_Y)^2]\leq c_5, \quad {\bE_{\omega}^y} [L_y^+(T_Y)^3]\leq c_6.\]
Therefore, by taking expectations in both sides of  \eqref{yuu1}, and using strong Markov property, we have 
\begin{align*}
\bE_{\omega}^y[L_y^+] &\leq {\bE_{\omega}^y}[L_y^+(T_Y)] + {\bP_{\omega}^y(T_Y<\infty)} \bP_{\omega}^Y(T_y <\infty)\bE_{\omega}^y[L_y^+] \\
&\leq {\bE_{\omega}^y}[L_y^+(T_Y)] +  \bP_{\omega}^Y(T_Y^+ <\infty)\bE_{\omega}^y[L_y^+] .
\end{align*}
After rearranging,
$$
\bE_{\omega}^y[L_y^+] \leq \frac{{\bE_{\omega}^y}[L_y^+(T_Y)]}{1-\bP_{\omega}^Y(T_Y^+ <\infty)}\leq \frac{{\bE_{\omega}^y}[L_y^+(T_Y)]}{\gamma_Y (\omega)}\leq \frac{3}{\gamma_Y (\omega)}.
$$
As for the second moment, we have
\begin{align*}
\bE_{\omega}^y[(L_y^+)^2]&\leq {\bE_{\omega}^y}[L_y^+(T_Y)^2] + 2{\bE_{\omega}^y}[L_y^+(T_Y)] \bP_{\omega}^Y(T_Y^+ <\infty)\bE_{\omega}^y[L_y^+]\\
&\quad 
+\bP_{\omega}^Y(T_Y^+ <\infty)\bE_{\omega}^y[(L_y^+)^2],
\intertext{and}
\bE_{\omega}^y[(L_y^+)^2]&\leq\frac{{\bE_{\omega}^y}[L_y^+(T_Y)^2] + 2{\bE_{\omega}^y}[L_y^+(T_Y)] \bP_{\omega}^Y(T_Y^+ <\infty)\bE_{\omega}^y[L_y^+]}{\gamma_Y(\omega)}\\
&\leq \frac{c_5 + 6 \cdot \frac{3}{\gamma_Y (\omega)}}{\gamma_Y(\omega)} \le { \frac{c_7}{\gamma_Y(\omega)^2}}.
\end{align*}
Turning to the third moment, we have
\begin{align*}
\bE_{\omega}^y[(L_y^+)^3]&\leq {\bE_{\omega}^y}[L_y^+(T_Y)^3] + 3{\bE_{\omega}^y}[L_y^+(T_Y)^2] \bP_{\omega}^Y(T_Y^+ <\infty)\bE_{\omega}^y[L_y^+]\\
&\quad +3{\bE_{\omega}^y}[L_y^+(T_Y)] \bP_{\omega}^Y(T_Y^+ <\infty)\bE_{\omega}^y[(L_y^+)^2]+\bP_{\omega}^Y(T_Y^+ <\infty)\bE_{\omega}^y[(L_y^+)^3],
\end{align*}
and
\begin{align*}
\bE_{\omega}^y[(L_y^+)^3]&\le \frac{c_6 + 3c_5 \bP_{\omega}^Y(T_Y^+ <\infty)\bE_{\omega}^y[L_y^+]+9 \bP_{\omega}^Y(T_Y^+ <\infty)\bE_{\omega}^y[(L_y^+)^2]}{\gamma_Y(\omega)}\\
&\le  \frac{c_8}{\gamma_Y(\omega)^3}.
\end{align*}
Recall that $\chi = 9/28$. {Using the Markov property and H\"{o}lder's inequality,}
\begin{align*}
\mathbf{E}[(L_y^+)^3]&=\E\bigl[ \bE_{\omega}[(L_y^+)^3] \bigr] = \E\bigl[ \bP_{\omega}(T_y < \infty)\bE_{\omega}^y[(L_y^+)^3] \bigr] \\
&\leq \E\left[{ \frac{c_8}{\gamma_Y(\omega)^3}}\right]\\
&=\sum_{z\in [\r, \parent{y}]} \E\left[\1_{\{Y=z\}}\1_{\{\omega(z, \parent{z})\leq 1-\epsilon\}}\left({\frac{c_8}{\gamma_z (\omega)^3}} \right)\right]\\
&\leq c_8 \sum_{z\in [\r, \parent{y}]}\P(Y=z)^{1-3\chi}   {\E\left[\left(\frac{\1_{\{\omega({z},\parent{z})\leq 1-\epsilon\}}}{\gamma_z(\omega)}\right)^{\frac{1}{\chi}}\right]^{3\chi},}
\end{align*}
For each fixed ray $\sigma = (\nu_i)_{i \in \N}$  where $\nu_{i+1} \sim \nu_i$ and $|\nu_{i+1} | = |\nu_{i}| +1$, we have that {the process $(\omega(\nu_{i+1},\nu_i))_{i \in \N}$ is composed by  \iid~random variables}. Hence 
{$\P(Y=z)\leq \P(\omega(\r1, \r)> 1-\epsilon)^{|y|-|z|}$} for any ancestor $z$ of $y$, and we have
\[ \mathbf{E}[(L_y^+)^3] \leq c_9 \sum_{n=0}^\infty {\P(\omega(\r1, \r)> 1-\epsilon)^{(1-3\chi)n}}  <\infty. \] 
This completes the proof.
\hspace{\fill}
\end{proof}
The next result is a by-product of Proposition~\ref{prop:Yuuma4-2} combined with the proof of  Proposition \ref{prop:Proposition15}.

\begin{proposition} \label{prop:Yuuma4-3}
{Consider LERRW on the 4-regular tree. For any $p \in (0, 3)$, we have $\mathbf{E}[\tau_1^{p}]<\infty$.}
\end{proposition}
\begin{proof}
{
It is enough to prove the  proposition for $p \in (1, 3)$. We have
\begin{equation}\label{na10}
\begin{aligned}
\bE[\tau_1^p] \le \bE\left[\left(\sum_{k=1}^\Pi L_{\sigma_k}\right)^p\right]\le \bE\left[ \Pi^{p-1} \sum_{k=1}^\Pi L_{\sigma_k}^p\right].
\end{aligned}
\end{equation}
Choose $t,q>1$  conjugates, i.e.   $(1/t)+ (1/q)=1$, and $tp<3$.  By H\"older's inequality, {and using Proposition~\ref{propls}}, we have that
\begin{equation}\label{na11}
\begin{aligned}
 \bE\left[ \Pi^{p-1} \sum_{k=1}^\Pi L_{\sigma_k}^p\right] &=\bE\left[ \sum_{k=1}^\infty L_{\sigma_k}^p  \Pi^{p-1} \1_{\{\Pi \ge k\}}\right]\\
 &\le \sum_{k=1}^\infty {\bE\left[L_{\sigma_k}^{pt}\right]^{1/t}} \bE\left[\Pi^{(p-1)q}\1_{\{\Pi \ge k\}}\right]^{1/q}\\
 &\le (C_{tp})^t \sum_{k=1}^\infty\bE\left[\Pi^{(p-1)q}\1_{\{\Pi \ge k\}}\right]^{1/q} <\infty,
 \end{aligned}
\end{equation}
 where $C_{tp}$ is the same as in Proposition~\ref{propls}.}
\hspace{\fill}
\end{proof}

\section{Appendix}
\begin{proposition}\label{prop:naap} {Let $p>0$.} Consider a random variable $Y$  geometrically distributed  with parameter $\theta \in (0,1)$,  and probability mass function
 $$ \bP(Y= k) = (1- \theta)^k \theta, \qquad \mbox{with $k \ge 0$}.$$
 If we set
 $\lambda \Def - \ln (1- \theta)$,  we have 
\begin{equation}\label{na6} 
 \bE[Y^p] \le  C_p \frac{\theta}{\lambda^{p+1}} +1,
\end{equation}
where $C_p$  is  a  positive finite constant that depends on $p$ but not on $\theta$.
\end{proposition}
\begin{proof}
 First  notice that the function $f \colon x \mapsto x^p {\rm e}^{- \lambda x}$, for $x \ge 0$, achieves its unique maximum at $x^* = p/\lambda$.  As $f$ is non-negative, and it is decreasing in the interval $[x^*, \infty)$, we have the following estimate
$$ \sum_{k=0}^\infty f(k) \le x^* f(x^*) + \int_0^\infty f(u) \d u.$$
Hence
$$
\begin{aligned}
 \bE[Y^p]  &=  \sum_{k=0}^\infty k^p (1-\theta)^{k} \theta \\
 &\le 1+  \theta\frac{p}{\lambda}\left(\frac{p}{\lambda}\right)^p {\rm e}^{-p}  +  \theta  \int_{0}^\infty  x^p {\rm e}^{- x \lambda} \d x\\
 &= 1+ \theta\frac{p}{\lambda}\left(\frac{p}{\lambda}\right)^p {\rm e}^{-p}+ \theta  \frac{\Gamma(p+1)}{\lambda^{p+1}}\\
 &=:  1+  C_p \frac{\theta}{\lambda^{p+1}}.
 \end{aligned}
$$ 
\hfill 
\end{proof}

\begin{ack}
{}
A.C. is grateful to Yokohama National University for its hospitality, and he was supported by ARC grant  DP180100613, Australian Research Council Centre of Excellence for Mathematical and Statistical Frontiers (ACEMS). CE140100049, {and YNU iROUTE project}. M.T. is partially supported by JSPS Grant-in-Aid for Young Scientists (B) No. 16K21039. {The authors thank the anonymous referee for detailed comments. Finally they thank Amanoya for offering a very nice environment, where part of this research was carried.}
\end{ack}


\begin{thebibliography}{99}
\bibitem{Aid} A\"id\'ekon,~E. (2008). Transient random walks in random environment on a Galton-Watson tree. {\it Probability Theory and Related Fields}      {\bf 142(3-4)}, 525-559.
{
\bibitem{Aid1} A\"id\'ekon,~E. (2010). Large deviations for transient random walks in random environment on a Galton--Watson tree. {\it Ann. Inst. H. Poincar\'e Probab. Statist.} {\bf 46(1)}, 159--189.}


\bibitem{CollBarb} 
Barbour,~A. and Collevecchio,~A. (2017).  General random walk in a random environment defined on Galton--Watson trees. {\it Ann. Inst. H. Poincar\'{e} Probab. Statist.} {\bf 53(4)}, 1657--1674.

\bibitem{Bog}
Bogachev, L. V. (2006). Random walks in random environments. {\it Encyclopedia of Mathematical Physics} {\bf 4}, 353--371.
     
\bibitem{Coll06a}
Collevecchio,~A. (2006). Limit theorems for reinforced random walks on certain trees.
{\it Probability Theory and Related Fields} {\bf 136(1)}, 81--101.

 
\bibitem{Coll09} Collevecchio,~A. (2009). Limit theorems for vertex-reinforced jump processes on regular trees. {\it Electron. J. of Probab.}  {\bf 14(66)}, 1936--1962.



\bibitem{CKS}
Collevecchio, A.,  Kious, D., and Sidoravicius, V. (2018). The branching-ruin number and the critical parameter of once-reinforced random walk on trees.  {\it Communications on Pure and Applied Mathematics}, Forthcoming.
{
\bibitem{ColS}
  Collevecchio, A., Schmitz, T. (2011). Bounds on the speed and on regeneration times for certain processes on regular trees. {\it Ann. Appl. Probab.}
    {\bf 21(3)}, 1073--1101.}

\bibitem{CD} Coppersmith, D.~and Diaconis, P.~(1986). Random walks with reinforcement. {\it Unpublished manuscript.}

\bibitem{Dav90}
Davis, B. (1990).
\newblock Reinforced random walk. 
\newblock {\it Probability Theory and Related Fields} {\bf 84}, 203--229.

 

\bibitem{DKL02}
Durrett, R., Kesten, H., and Limic, V. (2002).
\newblock Once edge-reinforced random walk on a tree,
\newblock {\it Probability Theory and Related Fields} {\bf 122(4)}, 567--592.

\bibitem{Kes1} 
Kesten, H., Kozlov, M. V.,  and Spitzer, F.  (1975). A limit law for random walk in a
random environment, {\it Comp. Math.} {\bf 30}, 145--168.


\bibitem{LyoPem} Lyons, R.~and Pemantle, R. (1992). Random walk in a random environment and first-passage percolation on trees. 
{\it Ann. Probab.}
{\bf 20(1)}, 125--136.



  \bibitem{Pemtree} Pemantle, R. (1988). Phase transition in reinforced random walk and RWRE on trees. {\it Ann. Probab.} {\bf 16}, 1229--1241.
   \bibitem{Per} 
{Peres, Y. and Zeitouni, O. (2008). A Central Limit Theorem for biased random walks on Galton-Watson trees.  {\it Probab. Theory Relat. Fields} {\bf 14(3--4)}, 595--629.}

  \bibitem{ST} Sabot,~C. and Tarr\`es,~P. (2015). Edge-reinforced random walk, vertex-reinforced jump process and the supersymmetric hyperbolic sigma model. {\it J. Eur. Math. Soc.} {\bf 17(9)}, 2353--2378.

\bibitem{Sol}
Solomon,  F.  (1975). Random walks in random environments. {\it Ann. Probab.} {\bf 3},
1--31.

\bibitem{Zei} Zeitouni,~O. (2004). Random walks in random environment. {\it Lecture Notes in Math.} {\bf 1837}, 189--312, Springer.
     
\bibitem{Zhang} Zhang,~Y. (2014). Large deviations in the reinforced random walk model on trees. {\it Probability Theory and Related Fields}  {\bf 160(3--4)}, 655--678.

 
\end{thebibliography}
\end{document}